\documentclass[reqno,12pt]{amsart}
\usepackage{amsfonts}
\usepackage{upgreek}
\usepackage{bbm}
\usepackage{} 
\numberwithin{equation}{section}
\usepackage{indentfirst}
\usepackage{color}
\usepackage{mathrsfs}
\usepackage[all]{xy}
\usepackage{tikz}
\xyoption{all}
 \def\Hom{\mbox{\rm Hom}} \def\silt{\mbox{\rm silt}\,}

\theoremstyle{plain} 
\newtheorem{theorem}{\bf Theorem}[section]
\newtheorem{lemma}[theorem]{\bf Lemma}
\newtheorem{corollary}[theorem]{\bf Corollary}
\newtheorem{proposition}[theorem]{\bf Proposition}

\theoremstyle{definition} 
\newtheorem{definition}[theorem]{\bf Definition}
\newtheorem{remark}[theorem]{\bf Remark}
\newtheorem{example}[theorem]{\bf Example}

\newcommand{\bt}{\begin{theorem}}
\newcommand{\et}{\end{theorem}}
\newcommand{\bl}{\begin{lemma}}
\newcommand{\el}{\end{lemma}}
\newcommand{\bd}{\begin{definition}}
\newcommand{\ed}{\end{definition}}
\newcommand{\bc}{\begin{corollary}}
\newcommand{\ec}{\end{corollary}}
\newcommand{\bp}{\begin{proof}}
\newcommand{\ep}{\end{proof}}
\newcommand{\bx}{\begin{example}}
\newcommand{\ex}{\end{example}}
\newcommand{\br}{\begin{remark}}
\newcommand{\er}{\end{remark}}
\newcommand{\be}{\begin{equation}}
\newcommand{\ee}{\end{equation}}
\newcommand{\ba}{\begin{align}}
\newcommand{\ea}{\end{align}}
\newcommand{\bn}{\begin{enumerate}}
\newcommand{\en}{\end{enumerate}}
\newcommand{\bcs}{\begin{cases}}
\newcommand{\ecs}{\end{cases}}

\makeatletter
\renewcommand{\section}{\@startsection{section}{1}{0mm}
  {-\baselineskip}{0.5\baselineskip}{\bf\leftline}}
\makeatother

\begin{document}

\title[Silting subcategories and (co)torsion pairs]{Silting subcategories and (co)torsion pairs\\ associated to extended hearts} 

\author{Liangwei Huang and Haicheng Zhang}
\address{Ministry of Education Key Laboratory of NSLSCS, School of Mathematical Sciences, Nanjing Normal University, Nanjing 210023, P.R.China}
\email{2524416777@qq.com (L. Huang)}
\email{zhanghc@njnu.edu.cn (H. Zhang)}

%
\keywords{ 
silting subcategories; $s$-torsion pairs; cotorsion pairs; extended hearts.
}


\begin{abstract}
We establish the poset isomorphisms between $(d+1)$-term silting subcategories,
functorially finite $s$-torsion pairs in the $d$-extended heart,
and hereditary complete cotorsion pairs in a suitable subcategory. As an application,
we also give dg algebra versions of these bijections, which establish the poset isomorphisms between $\tau$-tilting pairs, $(d+1)$-term silting complexes, and functorially finite $s$-torsion pairs.
\end{abstract}

\maketitle

\section{Introduction}
Support $\tau$-tilting modules were introduced by Adachi, Iyama and
Reiten \cite{AIR14} to complete the classical tilting theory via mutations.
Furthermore, they established bijections between support $\tau$-tilting pairs,
functorially finite torsion pairs, and $2$-term silting complexes.
Pauksztello and Zvonareva \cite{PZ23} gave a bijection between functorially finite torsion pairs and complete cotorsion pairs in a certain $2$-term category.
In order to generalize these bijections to silting complexes of arbitrary length, the
$d$-extended module categories were introduced
in \cite{G24, Z24}. More general, the $d$-extended heart was introduced in \cite{Z24}, which is a natural generalization of the hearts of $t$-structures in triangulated categories and forms an extriangulated category in the sense of \cite{NP19}.
{The $s$-torsion pairs in extriangulated categories were introduced in \cite{AET23} with the aim to provide a general framework for the study of $t$-structures in triangulated categories and torsion pairs in abelian categories.}
Recently, Gupta \cite{G24} generalized the bijections above to a more general framework,
and established the poset isomorphisms between $(d+1)$-term silting subcategories, functorially finite $s$-torsion pairs in $d$-extended hearts,
and hereditary complete cotorsion pairs in appropriate subcategories.
Zhou \cite{Z24} proved that $d$-extended hearts admit almost split extensions, introduced $\tau$-tilting pairs in $d$-extended hearts, and
generalized the bijections above to $d$-extended hearts, thereby expanding the application scope of tilting theory.

In addition to the case of finite-dimensional algebras, Wei and Zhou \cite{WZ25} extended their research perspective to more general triangulated categories. They introduced the concept of AIR tilting subcategories as generalizations of support $\tau$-tilting modules,
and established a bijection between AIR tilting subcategories and $(d+1)$-term silting subcategories under appropriate conditions.

In this paper, we generalize Gupta's bijections by establishing the poset isomorphisms between $(d+1)$-term silting subcategories,
functorially finite $s$-torsion pairs in $d$-extended hearts, and hereditary complete cotorsion pairs in specific subcategories under a more general Krull-Schmidt triangulated category framework. In the framework of non-positive dg algebras, Mochizuki and Plogmann [13] proved that each $d$-extended heart admits ARS duality in the sense of [10], provided that the dg algebra is contained in its $d$-extended heart as a regular dg module. This enables us to extend the $\tau$-tilting theory to derived categories of non-positive dg algebras and obtain a dg analogue of the bijections above.

Throughout the paper, all categories are additive categories. Let $\mathcal{C}$ be a category, for any collection $\mathcal{X}$ consisting of some objects in $\mathcal{C}$, we denote by $\operatorname{add}\mathcal{X}$ the {\em additive closure} of $\mathcal{X}$ in $\mathcal{C}$,
which is defined to be the full subcategory of $\mathcal{C}$ consisting of direct summands of finite direct sums of objects in $\mathcal{X}$. We write $\operatorname{add}(X)$ as $\operatorname{add}\mathcal{X}$ if $\mathcal{X}$ consists of a single object $X$.
When we say that $\mathcal{X}$ is a subcategory of $\mathcal{C}$,
we always assume that $\mathcal{X}$ is full and satisfies $\mathcal{X} =\operatorname{add}\mathcal{X}$.
For any two subcategories $\mathcal{X}$ and $\mathcal{Y}$ of $\mathcal{C}$,
we write $\operatorname{Hom}_{\mathcal{C}}(\mathcal{X},\mathcal{Y}) = 0$ to mean that $\operatorname{Hom}_{\mathcal{C}}(X,Y) = 0$ for any $X \in \mathcal{X}$ and $Y \in \mathcal{Y}$.
The analogous conventions apply to the notations $\operatorname{Hom}_{\mathcal{C}}(\mathcal{X},Y) = 0$ and $\operatorname{Hom}_{\mathcal{C}}(X,\mathcal{Y}) = 0$ for any fixed objects $X$ and $Y$ in $\mathcal{C}$. For the sake of simplicity in notation, we may write $\operatorname{Hom}(X,Y)$ as $\operatorname{Hom}_{\mathcal{C}}(X,Y)$
for any $X,Y \in \mathcal{C}$.

\section{Preliminaries}
In this section, we set some notations and recall some fundamental notions. Let $\mathcal{C}$ be a triangulated category with the suspension functor denoted by $[1]$.
Let $I$ be a subset of the integer set $\mathbb{Z}$, which is often taken to be the subset denoted by $>n$, $<n$, $\geq n$, $\leq n$, $\neq n$ or $n$ with the obvious associated meaning. For any subcategory $\mathcal{X}$ of $\mathcal{C}$, we define the following classes
\begin{flalign*}
&\mathcal{X}[I]:=\{X[i]\,|\,X\in\mathcal{X}\text{ and }i\in I\},\\
&{^{\perp_I}\mathcal{X}} := \{Y \in \mathcal{C}\,|\,\operatorname{Hom}(Y,\mathcal{X}[I]) = 0\},\\
&\mathcal{X}^{\perp_I} := \{Y \in \mathcal{C}\,|\,\operatorname{Hom}(\mathcal{X},Y[I]) = 0\}.
\end{flalign*}
If $\mathcal{X}={\rm add}(X)$ for some object $X\in\mathcal{C}$, we write ${^{\perp_I}X}$ and $X^{\perp_I}$ for ${^{\perp_I}\mathcal{X}}$ and $\mathcal{X}^{\perp_I}$, respectively.
For any subcategories $\mathcal{X}$ and $\mathcal{Y}$ of $\mathcal{C}$, we denote by $\mathcal{X} \ast  \mathcal{Y}$
the class of all objects $Z$ admitting a triangle $X \to Z \to Y \to X[1]$
with $X \in \mathcal{X}$ and $Y \in \mathcal{Y}$.

In what follows, let $\mathcal{B}$ be an {\em extension-closed subcategory} of $\mathcal{C}$, which means that for any triangle $X \to Z \to Y \to X[1]$ with $X,Y\in\mathcal{B}$, we have $Z\in\mathcal{B}$. Then $\mathcal{B}$ is an extriangulated category in the sense of \cite[Definition 2.12]{NP19}.

\begin{definition}
A pair of subcategories $(\mathcal{T},\mathcal{F})$ of $\mathcal{B}$ is said to be a \textit{torsion pair} if
${\rm Hom}(\mathcal{T},\mathcal{F})=0~\text{and}~\mathcal{B}=\mathcal{T}\ast\mathcal{F}.$
In this case, the subcategory $\mathcal{T}$ is called the \textit{torsion class}, and $\mathcal{F}$ is called the \textit{torsion-free class}.
The torsion pair $(\mathcal{T},\mathcal{F})$ is called an \textit{$s$-torsion pair} if moreover ${\rm Hom}(\mathcal{T},\mathcal{F}[\leq -1])=0$. We denote by $\operatorname{tors}\mathcal{B}$ the poset of torsion classes in $\mathcal{B}$ under inclusion, and
denote by $s\operatorname{-tors}\mathcal{B}$ the poset of $s$-torsion classes in $\mathcal{B}$ under inclusion.
\end{definition}

\begin{remark}
The torsion pairs give rise to approximations. Explicitly,
let $(\mathcal{T},\mathcal{F})$ be a torsion pair in $\mathcal{B}$.
Then every $X\in\mathcal{B}$ admits a triangle
$${T\stackrel{f}{\longrightarrow}X\stackrel{g}{\longrightarrow}F\stackrel{}{\longrightarrow}T[1]}$$
  with $T\in\mathcal{T}$ and $F\in\mathcal{F}$, which implies $f$ is a right $\mathcal{T}$-approximation of $X$ and $g$ is a left $\mathcal{F}$-approximation of $X$. In particular, $\mathcal{T}$ is contravariantly finite and $\mathcal{F}$ is covariantly finite in $\mathcal{B}$. Moreover, $\mathcal{T}={^{\bot_0}\mathcal{F}}\cap\mathcal{B}$ and $\mathcal{F}=\mathcal{T}^{\bot_0}\cap\mathcal{B}$.
\end{remark}

\begin{definition}
A pair of subcategories $(\mathcal{U},\mathcal{V})$ is called a \textit{cotorsion pair} in $\mathcal{B}$ if
\begin{enumerate}
\item ${\rm Hom}(\mathcal{U},\mathcal{V}[1])=0;$
\item For any $B\in\mathcal{B}$, there exist two triangles
 \begin{equation*}
    V\stackrel{}{\longrightarrow} U\stackrel{}{\longrightarrow} B\stackrel{}{\longrightarrow} V[1] \text{ and } B\stackrel{}{\longrightarrow} V'\stackrel{}{\longrightarrow} U'\stackrel{}{\longrightarrow} B[1]
 \end{equation*}
such that $U,\,U'\in\mathcal{U}$ and $V,\,V'\in\mathcal{V}$.
\end{enumerate}
In this case, the subcategory $\mathcal{V}$ is called the \textit{cotorsion class}.
We denote by $\operatorname{cotors}\mathcal{B}$ the poset of cotorsion classes in $\mathcal{B}$ under inclusion.

The cotorsion pair and the corresponding cotorsion class are said to be \textit{hereditary} if moreover ${\rm Hom}(\mathcal{U},\mathcal{V}[\geq 2])=0$.
We denote by $\operatorname{h.cotors}\mathcal{B}$ the poset of hereditary cotorsion classes in $\mathcal{B}$ under inclusion.
\end{definition}

\begin{remark}
  Let $(\mathcal{U},\mathcal{V})$ be a cotorsion pair in $\mathcal{B}$.
  Similarly to torsion pairs, we have that $\mathcal{U}$ is contravariantly finite and $\mathcal{V}$ is covariantly finite in $\mathcal{B}$.
  Moreover, $\mathcal{U}={^{\bot_1}\mathcal{V}}\cap\mathcal{B}$ and $\mathcal{V}=\mathcal{U}^{\bot_1}\cap\mathcal{B}$.
\end{remark}

\begin{definition}
A torsion pair $(\mathcal{T},\mathcal{F})$ in $\mathcal{C}$ is said to be
\begin{enumerate}
\item a \textit{$t$-structure} if $\mathcal{T}[1]\subseteq\mathcal{T}$, in this case, $\mathcal{T}\cap\mathcal{F}[1]$ is called the {\em heart} of $(\mathcal{T},\mathcal{F})$;
\item a \textit{co-$t$-structure} if $\mathcal{T}[-1]\subseteq\mathcal{T}$, in this case, $\mathcal{T}[1]\cap\mathcal{F}$ is called the {\em coheart} of $(\mathcal{T},\mathcal{F})$;
\item \textit{nondegenerate} if $\bigcap_{k\in \mathbb{Z}}\mathcal{T}[k]=\{0\}=\bigcap_{k\in \mathbb{Z}}\mathcal{F}[k]$;
\item \textit{bounded} if $\bigcup_{k\in \mathbb{Z}}\mathcal{T}[k]=\mathcal{C}=\bigcup_{k\in \mathbb{Z}}\mathcal{F}[k]$.
\end{enumerate}
\end{definition}

\begin{remark}\label{tsctc}
$(1)$ By \cite[Lemma 3.3]{AET23}, a torsion pair $(\mathcal{T},\mathcal{F})$ in $\mathcal{C}$ is a $t$-structure if and only if
  it is an $s$-torsion pair in $\mathcal{C}$.
  We usually denote a $t$-structure in $\mathcal{C}$ by $(\mathcal{C}^{\leq 0},\mathcal{C}^{\geq 1})$,
  and set
  \[(\mathcal{C}^{\leq n},\mathcal{C}^{\geq n+1}):=(\mathcal{C}^{\leq 0}[-n],\mathcal{C}^{\geq 1}[-n])\]
  for any $n\in\mathbb{Z}$.
  For any $X\in\mathcal{C}$, there exists a unique triangle up to isomorphism
  $${\tau_{\leq n}X\stackrel{}{\longrightarrow} X\stackrel{}{\longrightarrow}\tau_{\geq n+1}X\stackrel{}{\longrightarrow}\tau_{\leq n}X[1]}$$
  such that $\tau_{\leq n}X\in\mathcal{C}^{\leq n}$ and $\tau_{\geq n+1}X\in\mathcal{C}^{\geq n+1}$, which is called the \textit{$t$-decomposition of $X$ with respect to the $t$-structure $(\mathcal{C}^{\leq n},\mathcal{C}^{\geq n+1})$}.
  This $t$-decomposition induces  truncated functors
  \[\tau_{\leq n}:\mathcal{C}\longrightarrow\mathcal{C}^{\leq n} \text{ and } \tau_{\geq n+1}:\mathcal{C}\longrightarrow\mathcal{C}^{\geq n+1}\]
  such that
  $(\sigma_{\leq n},\tau_{\leq n})$ and $(\tau_{\geq n+1},\sigma_{\geq n+1})$ are adjoint pairs,
  where $\sigma_{\leq n}:\mathcal{C}^{\leq n}\to\mathcal{C}$ and $\sigma_{\geq n+1}:\mathcal{C}^{\geq n+1}\to\mathcal{C}$ are the natural inclusion functors.

$(2)$ A torsion pair $(\mathcal{T},\mathcal{F})$ in $\mathcal{C}$ is a co-$t$-structure if and only if
  $(\mathcal{T}[1],\mathcal{F})$ is a hereditary cotorsion pair in $\mathcal{C}$.

\end{remark}

\begin{definition}
  Let $\mathcal{S}$ be a subcategory of $\mathcal{C}$.

$(1)$ The subcategory $\mathcal{S}$ is called \textit{presilting} if ${\rm Hom}(\mathcal{S},\mathcal{S}[>0])=0$.

$(2)$ The subcategory $\mathcal{S}$ is called \textit{silting} if it is presilting and ${\bf thick}_{\mathcal{C}}(\mathcal{S})=\mathcal{C}$,
  where ${\bf thick}_{\mathcal{C}}(\mathcal{S})$ denotes the smallest triangulated subcategory of $\mathcal{C}$ which contains $\mathcal{S}$ and is closed under direct summands.
\end{definition}

Let us define a partial order on the collection $\operatorname{silt}\mathcal{C}$ of silting subcategories of $\mathcal{C}$ as follows:
\begin{equation}\label{siltingorder}
\text{For~any}~\mathcal{S}_1,\mathcal{S}_2\in\operatorname{silt}\mathcal{C},~\text{define}~\mathcal{S}_1\leq \mathcal{S}_2\Longleftrightarrow{\rm Hom}(\mathcal{S}_2,\mathcal{S}_1[>0])=0.\end{equation}
Let $\mathcal{S}$ be a silting subcategory of $\mathcal{C}$. For any integers $p\leq q$, set
\begin{equation*}
  \begin{split}
    &\mathcal{S}^{[p,q]}:=\mathcal{S}[p]\ast\mathcal{S}[p+1]\ast\cdots\ast\mathcal{S}[q],\\
    &\mathcal{S}^{(-\infty,q]}:=\bigcup_{k = -q}^{+\infty} \mathcal{S}^{[-k,q]},\\
    &\mathcal{S}^{[p,+\infty)}:=\bigcup_{k = p}^{+\infty} \mathcal{S}^{[p,k]}.
  \end{split}
\end{equation*}

\begin{lemma}\label{siltiff}
  Let $\mathcal{S}_1,\mathcal{S}_2$ be presilting subcategories of $\mathcal{C}$ such that ${\bf thick}_{\mathcal{C}}(\mathcal{S}_1)={\bf thick}_{\mathcal{C}}(\mathcal{S}_2)$.
Then for any integers $p\leq q$, the following are equivalent:
\begin{enumerate}
\item $\mathcal{S}_2\subseteq \mathcal{S}_1^{[p,q]};$
\item $\mathcal{S}_1[q]\leq \mathcal{S}_2\leq\mathcal{S}_1[p];$
\item $\mathcal{S}_1^{\perp_{>0}}[q]\subseteq \mathcal{S}_2^{\perp_{>0}}\subseteq \mathcal{S}_1^{\perp_{>0}}[p]$.
\end{enumerate}
\end{lemma}

\begin{proof}
The equivalence of $(1)$ and $(2)$ is obtained by \cite[Lemma 3.6]{AMY19}.
The implication $(3)\Rightarrow (2)$ follows from $\mathcal{S}_1\subseteq \mathcal{S}_1^{\perp_{>0}}$ and $\mathcal{S}_2\subseteq \mathcal{S}_2^{\perp_{>0}}$.

  For $(1)\Rightarrow (3)$, since
  $\mathcal{S}_1^{\perp_{>0}}[q]=(\mathcal{S}_1[q])^{\perp_{>0}}=(\mathcal{S}_1[q-1])^{\perp_{>-1}}=\cdots=(\mathcal{S}_1[p])^{\perp_{>p-q}},$
  we have $\mathcal{S}_1^{\perp_{>0}}[q]\subseteq(\mathcal{S}_1^{[p,q]})^{\perp_{>0}}\subseteq\mathcal{S}_2^{\perp_{>0}}$.
  On the other hand, by \cite[Lemma 3.7]{AMY19}, we get $\mathcal{S}_1\subseteq \mathcal{S}_2^{[-q,-p]}$.
  Thus, $\mathcal{S}_2^{\perp_{>0}}[-p]\subseteq(\mathcal{S}_2^{[-q,-p]})^{\perp_{>0}}\subseteq\mathcal{S}_1^{\perp_{>0}}$.
  Hence, we finish the proof.
\end{proof}

\section{Bijections between silting subcategories, $s$-torsion and cotorsion pairs}
In this section, let $\mathcal{D}$ be a triangulated category with a presilting subcategory $\mathcal{P}$ such that
\[
t=(\mathcal{D}^{\leq 0}, \mathcal{D}^{\geq 1}):=(\mathcal{P}^{\perp_{>0}}, \mathcal{P}^{\perp_{\leq 0}})
\]
is a $t$-structure in $\mathcal{D}$ with the heart denoted by $\mathcal{H}^1_{\mathcal{P}}$.
Fix a positive integer $d$, according to \cite{Z24}, the \textit{$d$-extended heart} is defined to be
\[
\mathcal{H}_{\mathcal{P}}^{d}:=\mathcal{D}^{\leq 0}\cap\mathcal{D}^{\geq -d+1}.
\]
Note that both $\mathcal{D}^{\leq 0}$ and $\mathcal{D}^{\geq -d+1}$ are closed under extensions, then so is $\mathcal{H}_{\mathcal{P}}^{d}$.
Thus, $\mathcal{H}_{\mathcal{P}}^{d}$ forms an extriangulated category, whose biadditive functor $\mathbb{E}$ is defined by
\[
\mathbb{E}(X,Y) := \operatorname{Hom}(X, Y[1])
\]
for any $X,Y \in \mathcal{H}_{\mathcal{P}}^{d}$, and the sequence $$Y\stackrel{f}{\longrightarrow}Z\stackrel{g}{\longrightarrow}X\stackrel{\delta}\dashrightarrow$$
with $X,Y,Z \in \mathcal{H}_{\mathcal{P}}^{d}$ is an \textit{extriangle} in $\mathcal{H}_{\mathcal{P}}^{d}$ if
$${Y \stackrel{f}{\longrightarrow} Z \stackrel{g}{\longrightarrow} X \stackrel{\delta}{\longrightarrow} Y[1]}$$ is a triangle in $\mathcal{D}$. Similarly, $\mathcal{P}^{[0,d]}$ is also an extriangulated category.

Set $\mathcal{K}(\mathcal{P}):={\bf thick}_{\mathcal{D}}(\mathcal{P})$. By definition, $\mathcal{P}$ is a silting subcategory of $\mathcal{K}(\mathcal{P})$. Let $\operatorname{b.co-t-str}\mathcal{K}(\mathcal{P})$ be the poset of bounded co-$t$-structures in $\mathcal{K}(\mathcal{P})$ ordered by the inclusions of torsion-free classes. By \cite[Proposition 2.23]{AI12} and \cite[Corollary 5.9]{HSSS13}, we have the following.

\begin{proposition}\label{siltisobcot}
  There is an isomorphism of posets
  \begin{align*}
\chi : \silt\mathcal{K}(\mathcal{P}) \longrightarrow \operatorname{b.co-t-str}\mathcal{K}(\mathcal{P}),~
\mathcal{S}  \longmapsto (\mathcal{S}^{(-\infty,-1]},\mathcal{S} ^{[0,+\infty)}).
\end{align*}
Moreover, $\mathcal{S} ^{[0,+\infty)}=\mathcal{S}^{\perp_{>0}}\cap\mathcal{K}(\mathcal{P}).$
\end{proposition}
Let $(d+1)$-${\rm silt}\,\mathcal{P}$ be the poset of the silting subcategories of $\mathcal{K}(\mathcal{P})$ contained in $\mathcal{P}^{[0,d]}$ under the order defined as \eqref{siltingorder}. In what follows, each subcategory in $(d+1)$-${\rm silt}\,\mathcal{P}$ is called the $(d+1)$-term silting subcategory with respect to $\mathcal{P}$.
\begin{proposition}\label{dsisohc}
  There is an isomorphism of posets
  \begin{align*}
    \psi : (d+1)\operatorname{-silt}\mathcal{P} \longrightarrow \operatorname{h.cotors}\mathcal{P}^{[0,d]},~
    \mathcal{S}  \longmapsto \mathcal{S} ^{[0,+\infty)} \cap \mathcal{P}^{[0,d]}.
  \end{align*}
\end{proposition}

\begin{proof}
By Proposition \ref{siltisobcot} and Lemma \ref{siltiff}, the map $\chi$ restricts to a poset isomorphism:
\[
\chi : (d+1)\operatorname{-silt}\mathcal{P} \stackrel{\sim}{\longrightarrow} \operatorname{b.co-t-stfs}\,[0,d],
\]
where $\operatorname{b.co-t-stfs}\,[0,d]$ denotes the poset
\[
  \{\mathcal{Y}~|~({^{\perp_{0}}}\mathcal{Y}\cap\mathcal{K}(\mathcal{P}),\mathcal{Y}) \,\text{is a bounded co-}t\text{-structure in} \,\mathcal{K}(\mathcal{P})~\text{s.t.}~\mathcal{P}^{[d,+\infty)}\subseteq \mathcal{Y}\subseteq \mathcal{P}^{[0,+\infty)}\}
\]
under inclusion.
Note that $\mathcal{P}^{[0,d]}=\mathcal{P}^{(-\infty,d]}\cap\mathcal{P}^{[0,+\infty)}$.
Applying \cite[Corollary 3.9]{AT22} to $(\mathcal{P}^{(-\infty,d]},\mathcal{P}^{[d+1,+\infty)})$ and $(\mathcal{P}^{(-\infty,-1]},\mathcal{P}^{[0,+\infty)})$,
we obtain a poset isomorphism
\begin{align*}
\xi: \operatorname{cotors}\,[0,d] \stackrel{\sim}{\longrightarrow} \operatorname{cotors}\mathcal{P}^{[0,d]},~
\mathcal{\mathcal{Y}}\longmapsto \mathcal{Y} \cap \mathcal{P}^{[0,d]}
\end{align*}
with the inverse map $\eta$ defined by $\mathcal{\mathcal{V}}  \mapsto \operatorname{add}(\mathcal{V}\ast \mathcal{P}^{[d,+\infty)})$,
where $\operatorname{cotors}\,[0,d]$ denotes the poset
\[\{\mathcal{Y}~|~({^{\perp_{1}}}\mathcal{Y}\cap\mathcal{K}(\mathcal{P}),\mathcal{Y})~\text{is a cotorsion pair in} ~\mathcal{K}(\mathcal{P})
\text{ s.t.}~\mathcal{P}^{[d,+\infty)}\subseteq \mathcal{Y}\subseteq \mathcal{P}^{[0,+\infty)}\}\]
under inclusion and ${^{\perp_1}}\eta(\mathcal{V} )\cap\mathcal{K}(\mathcal{P})=\operatorname{add}(\mathcal{P}^{(-\infty,0]}\ast ({^{\perp_1}\mathcal{V}}\cap\mathcal{P}^{[0,d]}))$.

For any $(d+1)$-term silting subcategory $\mathcal{S}$, the image $\xi (\chi (\mathcal{S}))=\mathcal{S} ^{[0,+\infty)} \cap \mathcal{P}^{[0,d]}$ is hereditary. Denote by $\xi'$ the restriction of $\xi$ on $\operatorname{b.co-t-stfs}\,[0,d]$ and set $\psi=\xi'\circ\chi$, i.e. we have the following commutative diagram
\[
\xymatrix{
(d+1)\operatorname{-silt}\mathcal{P} \ar[rr]^-{\chi}_-{\cong} \ar[dr]_-{\psi}
& & \operatorname{b.co-t-stfs}\,[0,d]\ar@{>->}[dl]^-{\xi'}\\
& \operatorname{h.cotors}\mathcal{P}^{[0,d]}
}
\]

Let $(\mathcal{U}, \mathcal{V})$ be a hereditary cotorsion pair in $\mathcal{P}^{[0,d]}$. Then
\[(\operatorname{add}(\mathcal{P}^{(-\infty,0]}\ast \mathcal{U}), \operatorname{add}(\mathcal{V}\ast \mathcal{P}^{[d,+\infty)}))\]
is a cotorsion pair satisfying
\[
\mathcal{P}^{[d,+\infty)}\subseteq \operatorname{add}(\mathcal{V}\ast \mathcal{P}^{[d,+\infty)})\subseteq \mathcal{P}^{[0,+\infty)}
\]
and $\xi(\operatorname{add}(\mathcal{V}\ast\mathcal{P}^{[d,+\infty)}))= \xi(\eta(\mathcal{V}))=\mathcal{V}$.

By the heredity of $(\mathcal{U},\mathcal{V})$, we obtain \[
(\operatorname{add}(\mathcal{P}^{(-\infty,0]}\ast \mathcal{U}), \operatorname{add}(\mathcal{V}\ast \mathcal{P}^{[d,+\infty)}))
\]
is hereditary. Thus, by Remark \ref{tsctc} $(2)$, we conclude that \begin{equation}\label{ytjg}
(\operatorname{add}(\mathcal{P}^{(-\infty,0]}\ast \mathcal{U})[-1], \operatorname{add}(\mathcal{V}\ast \mathcal{P}^{[d,+\infty)}))\end{equation}
is a co-$t$-structure in $\mathcal{K}(\mathcal{P})$.
Since
\begin{align*}
  \bigcup_{n \in \mathbb{Z}} \mathrm{add}( \mathcal{P}^{(-\infty,0]} \ast \mathcal{U})[n] \supseteq \bigcup_{n \in \mathbb{Z}} \mathcal{P}^{(-\infty,n]} = \mathcal{K} (\mathcal{P}),\\
  \bigcup_{n \in \mathbb{Z}} \mathrm{add}( \mathcal{V} \ast \mathcal{P}^{[0,+\infty)})[n] \supseteq \bigcup_{n \in \mathbb{Z}} \mathcal{P}^{[n,+\infty)} = \mathcal{K} (\mathcal{P}),
\end{align*}
we get that the co-$t$-structure \eqref{ytjg} is bounded.
{Thus, we obtain
$$\mathrm{add}( \mathcal{V} \ast \mathcal{P}^{[0,+\infty)})\in \operatorname{b.co-t-stfs}\,[0,d]$$
and}
$$\xi'(\operatorname{add}(\mathcal{V}\ast\mathcal{P}^{[d,+\infty)}))=\xi(\operatorname{add}(\mathcal{V}\ast\mathcal{P}^{[d,+\infty)}))= \mathcal{V}.$$
That is, $\xi'$ is surjective and thus is an isomorphism. Hence, $\psi$ is also an isomorphism.
\end{proof}

We recall the following definition of $\mathcal{P}$-presentations for objects in $\mathcal{H}_{\mathcal{P}}^d$ from \cite{WZ25}.
\begin{definition}
For any object \(M \in \mathcal{H}_{\mathcal{P}}^d\),  a \textit{\(\mathcal{P}\)-presentation} of $M$ is an object $S\in\mathcal{P}^{[0,d]}$ satisfying $\tau_{\geq -d+1}S \cong M$.
A \textit{\(\mathcal{P}\)-presentation} of a subcategory \(\mathcal{M} \subseteq \mathcal{H}_{\mathcal{P}}^d\) is a subcategory $\mathcal{S}\subseteq\mathcal{P}^{[0,d]}$ satisfying \(\mathrm{add}(\tau_{\geq -d+1}\mathcal{S}) = \mathcal{M}\),
where $\tau_{\geq -d+1}\mathcal{S}:=\{\tau_{\geq -d+1}S \mid S\in\mathcal{S}\}$.
\end{definition}

\begin{lemma}\cite[Lemma 2.3]{WZ25}\label{Ppres}
  If \( \mathcal{P}\) is contravariantly finite in \(\mathcal{D}\), every object $M\in\mathcal{H}_{\mathcal{P}}^{d}$ admits a $\mathcal{P}$-presentation.
\end{lemma}

Now, we can show that the truncation functor induces an equivalence of additive categories.
This is a generalization of \cite[Proposition 3.1]{G24} and \cite[Theorem 2.5]{MP25}.

\begin{proposition}\label{PequH}
If $\mathcal{P}$ is contravariantly finite in $\mathcal{D}$,
then the functor
\[\tau_{\geq -d+1}:\mathcal{P}^{[0,d]}\stackrel{}{\longrightarrow}\mathcal{H}_{\mathcal{P}}^d\]
induces an equivalence of additive categories
\begin{equation}\label{tauhz}
\tau_{\geq -d+1}:\ \frac{\mathcal{P}^{[0,d]}}{\mathcal{P}[d]} \stackrel{\sim}{\longrightarrow} \mathcal{H}_{\mathcal{P}}^d.
\end{equation}
\end{proposition}

\begin{proof}
For any $R,\,Q\in\mathcal{P}^{[0,d]}$, consider the $t$-decompositions of $R$ and $Q$:
\begin{flalign}
&\tau_{\leq -d}R \stackrel{}{\longrightarrow} R\stackrel{u}{\longrightarrow} \tau_{\geq -d+1}R \stackrel{}{\longrightarrow} \tau_{\leq -d}R[1],\label{diyifj}\\
&\tau_{\leq -d}Q \stackrel{}{\longrightarrow} Q\stackrel{v}{\longrightarrow} \tau_{\geq -d+1}Q \stackrel{}{\longrightarrow} \tau_{\leq -d}Q[1]\label{dierfj}.
\end{flalign}
Applying $\operatorname{Hom}(R,-)$ to the triangle \eqref{dierfj},
we obtain an exact sequence
\[
\operatorname{Hom}(R,Q) \stackrel{v_\ast}{\longrightarrow} \operatorname{Hom}(R,\tau_{\geq -d+1}Q) \stackrel{}{\longrightarrow} \operatorname{Hom}(R,\tau_{\leq -d}Q[1]).
\]
Since $\mathcal{D}^{\leq -d-1}=\mathcal{P}^{\perp_{\geq -d}}\subseteq(\mathcal{P}^{[0,d]})^{\perp_{0}}$, we have $\operatorname{Hom}(R,\tau_{\leq -d}Q[1]) = 0$.
Thus, the map $v_\ast$
is surjective. Note that we have the isomorphism $$\operatorname{Hom}(\tau_{\geq -d+1}R,\tau_{\geq -d+1}Q) \stackrel{u^\ast}{\longrightarrow} \operatorname{Hom}(R,\tau_{\geq -d+1}Q)$$ by the adjunction $\tau_{\geq-d+1}\dashv \sigma_{\geq-d+1} $. It is easy to see that $(u^\ast)^{-1}(v_\ast(f))=\tau_{\geq -d+1}(f)$ for any $f\in\operatorname{Hom}(R,Q)$.
Hence, we conclude that $\tau_{\geq -d+1}$ is full.

Since $\mathcal{P}[d] \subseteq \mathcal{P}^{\perp_{>-d}} = \mathcal{D}^{\leq -d}$, we have $\tau_{\geq -d+1}(f)=0$
for any $f\in\operatorname{Hom}(R,Q)$ which factors through $\mathcal{P}{[d]}$.
Conversely, let $f\in\Hom(R,Q)$ such that $\tau_{\geq -d+1}(f) = 0$.
By the definition of $\mathcal{P}^{[0,d]}$, there is a triangle
\begin{equation}\label{fjzhl}
R^{[0,d-1]} \stackrel{\iota}{\longrightarrow} R \stackrel{\pi}{\longrightarrow} R^{d}[d]\stackrel{}{\longrightarrow} R^{[0,d-1]}[1]
\end{equation}
such that $R^d \in \mathcal{P}$ and $R^{[0,d-1]} \in \mathcal{P}^{[0,d-1]}$.
Applying the functor $\operatorname{Hom}(-,Q)$ to the triangle \eqref{fjzhl}, we get an exact sequence
\[
\operatorname{Hom}(R^{d}[d],Q)\stackrel{\pi^\ast}{\longrightarrow} \operatorname{Hom}(R,Q)\stackrel{\iota^\ast}{\longrightarrow}\operatorname{Hom}(R^{[0,d-1]},Q).
\]
Considering the $t$-decompositions of $R^{[0,d-1]}$ and $Q$ with respect to $(\mathcal{D}^{\leq-d},\mathcal{D}^{\geq-d+1})$, we obtain a commutative diagram of triangles
\[
\xymatrix{
\tau_{\leq-d}R^{[0,d-1]} \ar[r]\ar[d] & R^{[0,d-1]} \ar[d]^{f\circ \iota}\ar[r] & \tau_{\geq-d+1}R^{[0,d-1]}\ar[d]^{\tau_{\geq -d+1}(f\circ \iota)}\ar[r]&\tau_{\leq-d}R^{[0,d-1]}[1]\ar[d] \\
\tau_{\leq-d}Q\ar[r]&Q  \ar[r]^-{p}&\tau_{\geq -d+1}Q\ar[r]&\tau_{\leq-d}Q[1].
}
\]
Applying $\operatorname{Hom}(R^{[0,d-1]}, -) $ to the lower triangle,
we have an exact sequence
\[
\operatorname{Hom}(R^{[0,d-1]}, \tau_{\le -d} Q)\stackrel{}{\longrightarrow} \operatorname{Hom}(R^{[0,d-1]}, Q) \stackrel{p_\ast}{\longrightarrow} \operatorname{Hom}(R^{[0,d-1]}, \tau_{\ge -d+1} Q).
\]
Since
$\mathcal{D}^{\le -d}\subseteq (\mathcal{P}^{[0,d-1]})^{ \perp_{0}},$ we have $\operatorname{Hom}(R^{[0,d-1]}, \tau_{\le -d} Q)=0$.
{Thus, $p_\ast$ is injective.}
Noting that $\tau_{\geq -d+1}(f\circ \iota)=\tau_{\geq -d+1}(f)\circ \tau_{\geq -d+1}(\iota)=0,$
we get
$p_\ast(f\circ \iota)=0$.
It follows that $f\circ\iota=0$, {and then} $f$ factors through $\pi$. Thus, $f$ factors through $\mathcal{P}[d]$.
Hence, the functor $\tau_{\geq -d+1}:\mathcal{P}^{[0,d]}\rightarrow\mathcal{H}_{\mathcal{P}}^d$ induces the fully faithful functor \eqref{tauhz},
which is dense by Lemma \ref{Ppres}. Therefore, we complete the proof.
\end{proof}

As an immediate consequence of the above proposition, we have the following corollary,
which is a generalization of \cite[Lemma 3.8]{G24}.

\begin{corollary}\label{covfequ}
Assume that $\mathcal{P}$ is contravariantly finite in $\mathcal{D}$,
and let $\mathcal{V}$ be a subcategory of $\mathcal{P}^{[0,d]}$.
If $\mathcal{V}$ is covariantly finite in $\mathcal{P}^{[0,d]}$,
then $\tau_{\geq -d+1}\mathcal{V}$ is covariantly finite in $\mathcal{H}_{\mathcal{P}}^d$.
In addition, if $\mathcal{P}{[d]} \subseteq \mathcal{V}$, then the converse holds.

\end{corollary}
\begin{proof}
The first statement follows directly from the fact that the functor $\tau_{\geq -d+1}$ is full and dense.
Now suppose $\tau_{\geq -d+1}\mathcal{V}$ is covariantly finite in $\mathcal{H}_{\mathcal{P}}^d$.
For any $R \in \mathcal{P}^{[0,d]}$, let
\[
\tau_{\geq -d+1}f: \tau_{\geq -d+1}R \stackrel{}{\longrightarrow} \tau_{\geq -d+1}V
\]
be a left $\tau_{\geq -d+1}\mathcal{V}$-approximation with $V\in\mathcal{V}$.

By the structure of $\mathcal{P}^{[0,d]}$, there is a triangle
\[
R^{[0,d-1]} \stackrel{\iota}{\longrightarrow} R \stackrel{\pi}{\longrightarrow} R^d[d]\stackrel{}{\longrightarrow} R^{[0,d-1]}[1]
\]
such that $R^d \in \mathcal{P}$ and $R^{[0,d-1]} \in \mathcal{P}^{[0,d-1]}$. Since $\mathcal{P}{[d]} \subseteq \mathcal{V}$, we have $R^d[d]\in\mathcal{V}$.

For any  $f': R \to V'$ with $V' \in \mathcal{V}$, since $\tau_{\geq -d+1}f$ is a left $\tau_{\geq -d+1}\mathcal{V}$-approximation,
there exists $\alpha: V \to V'$ such that
$\tau_{\geq -d+1}(\alpha) \circ \tau_{\geq -d+1}(f) = \tau_{\geq -d+1}(f').$
By Proposition \ref{PequH}, there exist morphisms $g: R\to Q^d[d]$ and $h:Q^d[d]\to V'$ with $Q^d[d] \in \mathcal{P}[d]$ such that $f'-\alpha f=h\circ g$.
Since $\operatorname{Hom}(R^{[0,d-1]}, Q^d[d]) = 0$ and then $g\circ \iota=0$,
there exists $k:R^d[d]\rightarrow Q^d[d]$ such that $g=k\circ\pi$, i.e. we have
the following commutative diagram
\[
\xymatrix{
&R \ar[rr]^-{f' - \alpha f} \ar[dr]_-{g} \ar[dl]_-{\pi} & & V' \\
R^d[d] \ar[rr]_-{k} && Q^d[d]. \ar[ur]_-{h}
}
\]
Thus, we get the following commutative diagram
\[
\xymatrix{
R \ar[rr]^-{\left(\begin{smallmatrix} f \\ \pi \end{smallmatrix}\right)} \ar[dr]_-{f'} & & V \oplus R^d[d] \ar[dl]^-{(\alpha \ hk)} \\
& V'
}
\]
Hence, ${\left(\begin{smallmatrix} f \\ \pi \end{smallmatrix}\right)}: R \rightarrow V \oplus R^d[d]$
is a left $\mathcal{V}$-approximation of $R$.
Therefore, $\mathcal{V} $ is covariantly finite in $\mathcal{P}^{[0,d]}$.
\end{proof}

Let $(d+1)$-${\rm conf. silt}\,\mathcal{P}$ be the poset consisting of the silting subcategories of $\mathcal{K}(\mathcal{P})$ which are contained in $\mathcal{P}^{[0,d]}$ and contravariantly finite in $\mathcal{D}$.
Let ${\rm f.}s\operatorname{-tors}\mathcal{H}_{\mathcal{P}}^{d}$ denote the poset of functorially finite $s$-torsion classes in $\mathcal{H}_{\mathcal{P}}^{d}$.

\begin{proposition}\label{monofs}
If $\mathcal{P}$ is contravariantly finite in $\mathcal{D}$, there is a monomorphism of posets
\begin{align*}
  \phi: (d+1)\operatorname{-conf.silt}\mathcal{P} \longrightarrow {\rm f.}s\operatorname{-tors}\mathcal{H}_{\mathcal{P}}^{d},~
  \mathcal{S}  \longmapsto \mathcal{S}^{\perp_{>0}} \cap \mathcal{H}_{\mathcal{P}}^d.
\end{align*}
\end{proposition}

\begin{proof}
According to \cite[Proposition 2.5]{WZ25} and its proof, we have
\[
\phi: (d+1)\operatorname{-conf.silt}\mathcal{P} \longrightarrow s\operatorname{-tors}\mathcal{H}_{\mathcal{P}}^{d}
\]
is a monomorphism of posets.
It suffices to show that $\phi (\mathcal{S})$ is covariantly finite in $\mathcal{H}_\mathcal{P}^d$.

By Proposition \ref{dsisohc},
$\psi (\mathcal{S})= \mathcal{S}^{\perp_{>0}} \cap \mathcal{P}^{[0,d]}$
is a hereditary cotorsion class in $\mathcal{P}^{[0,d]}$,
and thus covariantly finite in $\mathcal{P}^{[0,d]}$.

For any {$i>0$,} $S\in\mathcal{S}$ and $R \in \psi(\mathcal{S})$,
applying $\operatorname{Hom}(S, -)$ to the triangle
\[
\tau_{\le-d}R \stackrel{}{\longrightarrow} R\stackrel{}{\longrightarrow} \tau_{\ge-d+1}R \stackrel{}{\longrightarrow}  \tau_{\le-d}R[1],
\]
we obtain the following exact sequence
\[
0=\operatorname{Hom}(S,\tau_{\leq -d}R[i])\to \operatorname{Hom}(S, R[i]) \to\operatorname{Hom}(S, \tau_{\ge-d+1}R[i]) \to\operatorname{Hom}(S,\tau_{\leq -d}R[i+1])=0,
\]
where the first and last terms vanish, since $\mathcal{D}^{\le-d} = \mathcal{P}^{\perp_{>0}}[d] \subseteq \mathcal{S}^{\perp_{>0}}$.
Thus, we have $$\operatorname{Hom}(S, \tau_{\ge-d+1}R[i]) \cong \operatorname{Hom}(S, R[i])=0.$$
It follows that $\tau_{\ge-d+1}\psi (\mathcal{S})\subseteq \phi (\mathcal{S})$.
Using the same arguments and Lemma \ref{Ppres}, we obtain $\phi (\mathcal{S})\subseteq \tau_{\ge-d+1}\psi (\mathcal{S})$.
Hence, $\tau_{\ge-d+1}\psi (\mathcal{S}) = \phi (\mathcal{S})$.
By Corollary \ref{covfequ}, $\phi (\mathcal{S})$ is covariantly finite in $\mathcal{H}_{\mathcal{P}}^d$.
Therefore, we complete the proof.
\end{proof}

Let us recall the Wakamatsu lemma in extriangulated categories, which will be used in the proof of the main theorem.

\begin{lemma}\cite[Lemma 3.1]{LZ20}\label{Wakamatsu}
  Let $\mathcal{B}$ be an extension closed subcategory of a triangulated category $\mathcal{C}$.
  Let
  \[X\stackrel{b}{\longrightarrow} B \stackrel{}{\longrightarrow} C \stackrel{}{\longrightarrow} X[1]\] be a triangle
  such that the first three terms lie in $\mathcal{B}$ and
  $b$ is a minimal left $\mathcal{B}$-approximation of $X$, then $C \in {^{\perp_1}}\mathcal{B}$.
\end{lemma}

For any $f\in\operatorname{Hom}(X,Y)$, denote by $C(f)$ the cone of $f$.
The following technical lemma is a generalization of \cite[Lemma 3.7]{G24}.

\begin{lemma}\label{tauCtau}
For any morphism $f: R \to Q$ with $R,\,Q\in \mathcal{P}^{[0,d]}$, we have
\[
\tau_{\ge-d+1}C(f) \cong \tau_{\ge-d+1}C(\tau_{\ge-d+1}f).
\]
\end{lemma}

\begin{proof}
By the $3\times3$ lemma for triangulated categories, we have the following diagram in which all rows and columns are triangles
\[
\xymatrix{
\tau_{\le-d}R \ar[r] \ar[d]_-{\tau_{\le-d}f} & R \ar[r] \ar[d]^f & \tau_{\ge-d+1}R \ar[d]^{\tau_{\ge-d+1}f}\ar[r]& \tau_{\le-d}R[1]\ar[d]\\
\tau_{\le-d}Q \ar[r] \ar[d] & Q \ar[r] \ar[d] & \tau_{\ge-d+1}Q \ar[d]\ar[r]&\tau_{\le-d}Q[1]\ar[d] \\
C(\tau_{\le-d}f)\ar[d] \ar[r] & C(f) \ar[d]\ar[r] & C(\tau_{\ge-d+1}f)\ar[r]\ar[d]&C(\tau_{\le-d}f)[1]\ar[d]\\
\tau_{\le-d}R[1] \ar[r] & R[1] \ar[r] & \tau_{\ge-d+1}R[1] \ar[r]& \tau_{\le-d}R[2].
}
\]
By the first column, we get $C(\tau_{\le-d}f)\in\mathcal{D}^{\leq-d}$.
Applying the octahedral axiom to the triangle on the third row of the above diagram
and the triangle
\[\tau_{\le-d}C(\tau_{\ge-d+1}f)\to C(\tau_{\ge-d+1}f)\to \tau_{\ge-d+1}C(\tau_{\ge-d+1}f)\to\tau_{\le-d}C(\tau_{\ge-d+1}f)[1],\]
we obtain a commutative diagram of triangles
\[
\xymatrix{
C(f) \ar[r] \ar@{=}[d] & C(\tau_{\ge-d+1}f) \ar[r] \ar[d] & C(\tau_{\le-d}f)[1] \ar[d]\ar[r]&C(f)[1]\ar@{=}[d] \\
C(f) \ar[r] & \tau_{\ge-d+1}C(\tau_{\ge-d+1}f) \ar[r] \ar[d] & W[1] \ar[d]\ar[r]&C(f)[1]\ar[d]\\
& \tau_{\le-d}C(\tau_{\ge-d+1}f)[1]\ar[d] \ar@{=}[r] & \tau_{\le-d}C(\tau_{\ge-d+1}f)[1]\ar[d]\ar[r]&C(\tau_{\ge-d+1}f)[1]\\
&C(\tau_{\ge-d+1}f)[1]\ar[r]&C(\tau_{\leq-d}f)[2].&
}
\]
By the third column of the above diagram, we have $W[1]\in\mathcal{D}^{\leq-d-1}$. Thus, we get
\[
W \stackrel{}{\longrightarrow} C(f) \stackrel{}{\longrightarrow} \tau_{\ge-d+1}C(\tau_{\ge-d+1}f)\stackrel{}{\longrightarrow} W[1]
\]
is the $t$-decomposition of $C(f)$ with respect to $(\mathcal{D}^{\leq -d},\mathcal{D}^{\geq -d+1})$.
By the uniqueness of $t$-decomposition up to isomorphism, we finish the proof.
\end{proof}

We have the following characterization for the heredity of cotorsion pairs, which generalizes the results in \cite[Lemma 3.13]{G24}.

\begin{lemma}\label{hequs}
Let $(\mathcal{U},\mathcal{V})$ be a cotorsion pair in $\mathcal{P}^{[0,d]}$ and $\mathcal{P}[d]\subseteq \mathcal{V}$ .
Then the following are equivalent.
\begin{enumerate}
\item $(\mathcal{U},\mathcal{V})$ is hereditary;
\item $\operatorname{Hom}(\mathcal{U},\mathcal{V}[2])=0$;
\item $\mathcal{V}$ is closed under cones.
\end{enumerate}
\end{lemma}

\begin{proof}
The implication $(1)\Rightarrow (2)$ is clear, and
$(2)\Rightarrow (3)$ follows from \cite[Lemma 3.2]{AT22}.

For $(3)\Rightarrow (1)$, take $V\in\mathcal{V}$,
consider the triangle
\[
V^{[0,d-1]}\stackrel{}{\longrightarrow} V \stackrel{\pi_V}{\longrightarrow} V_0^d[d] \stackrel{}{\longrightarrow} V^{[0,d-1]}[1]
\]
with $V^{[0,d-1]}\in\mathcal{P}^{[0,d-1]}$ and $V_0^d[d]\in\mathcal{P}[d]\subseteq\mathcal{V}$.
Noting that $V^{[0,d-1]}[1]\in \mathcal{P}^{[0,d]}$, we have
\[
V \stackrel{\pi_V}{\longrightarrow} V_0^d[d] \stackrel{}{\longrightarrow} V^{[0,d-1]}[1]\dashrightarrow
\]
is an extriangle in $\mathcal{P}^{[0,d]}$.
Since $\mathcal{V}$ is closed under cones, we obtain $V^{[0,d-1]}[1]\in \mathcal{V}$.
Thus, we obtain a triangle
\[
V^{[0,d-1]}[1] \stackrel{}{\longrightarrow} V[1] \stackrel{}{\longrightarrow} V_0^d[d+1] \stackrel{}{\longrightarrow} V^{[0,d-1]}[2]
\]
with $V^{[0,d-1]}[1] \in \mathcal{V}$ and $V_0^d[d+1] \in \mathcal{P}[d+1]$.
Furthermore, for $V^{[0,d-1]}[1]$ we have the triangle
\[
V^{[1,d-1]}\stackrel{}{\longrightarrow} V^{[0,d-1]}[1] \stackrel{}{\longrightarrow} V_1^d[d] \stackrel{}{\longrightarrow} V^{[1,d-1]}[1]
\]
with $V^{[1,d-1]}\in\mathcal{P}^{[1,d-1]}$ and $V_1^d[d]\in\mathcal{P}[d]\subseteq\mathcal{V} $.
Also, since $\mathcal{V}$ is closed under cones, we get $V^{[1,d-1]}[1] \in \mathcal{V} $.
Applying the octahedral axiom yields the following commutative diagram of triangles
\[
\xymatrix{
V^{[1,d-1]} \ar@{=}[d] \ar[r] & V^{[0,d-1]}[1] \ar[r] \ar[d] & V_1^d[d] \ar[d]\ar[r]&V^{[1,d-1]}[1]\ar@{=}[d] \\
V^{[1,d-1]} \ar[r] & V[1] \ar[r] \ar[d] & V^{[d,d+1]} \ar[d]\ar[r]&V^{[1,d-1]}[1] \ar[d]\\
& V_0^d[d+1] \ar@{=}[r]\ar[d] & V_0^d[d+1]\ar[r]\ar[d]&V^{[0,d-1]}[2]\\
&V^{[0,d-1]}[2]\ar[r]&V_1^d[d+1].&
}
\]
Thus, $V^{[d,d+1]} \in \mathcal{P}^{[d,d+1]}$ and we have the following triangle
\[
V^{[1,d-1]}[1] \stackrel{}{\longrightarrow} V[2] \stackrel{}{\longrightarrow} V^{[d,d+1]}[1] \stackrel{}{\longrightarrow} V^{[1,d-1]}[2]
\]
with $V^{[1,d-1]}[1] \in \mathcal{V}$ and $V^{[d,d+1]}[1] \in \mathcal{P}^{[d+1,d+2]}$.

By induction, for each integer $0 \leq i \leq d-1$, we have the following triangle
\[
V^{[i,d-1]}[1] \stackrel{}{\longrightarrow} V[i+1] \stackrel{}{\longrightarrow} V^{[d,d+i]}[1] \stackrel{}{\longrightarrow} V^{[i,d-1]}[2]
\]
with $V^{[i,d-1]}[1] \in \mathcal{V}$ and $V^{[d,d+i]}[1] \in \mathcal{P}^{[d+1,d+i+1]}$.

For any $U \in \mathcal{U}$,
applying $\operatorname{Hom}(U, -)$ to the above triangles yields the exact sequences
\[
0=\operatorname{Hom}(U, V^{[i,d-1]}[2])
\stackrel{}{\longrightarrow} \operatorname{Hom}(U, V[i+2])
\stackrel{}{\longrightarrow} \operatorname{Hom}(U, V^{[d,d+i]}[2]) = 0.
\]
Note that $V[j+2]\in \mathcal{P}^{[d+2,+\infty)}\subseteq (\mathcal{P}^{[0,d]})^{\perp_0}$ for any $j \geq d$. So, $\operatorname{Hom}(U, V[\geq 2])=0$.
Hence,
$\operatorname{Hom}(\mathcal{U},\mathcal{V}[\geq 2])=0,$
i.e. the cotorsion pair $(\mathcal{U},\mathcal{V})$ is hereditary.
\end{proof}

Now, we are in a position to give the main theorem as follows.

\begin{theorem}\label{sfbij}
Let $\mathcal{D}$ be a Krull-Schmidt triangulated category. Assume that all $(d+1)$-term silting subcategories with respect to $\mathcal{P}$ are contravariantly finite in $\mathcal{D}$.
Then there are isomorphisms of posets such that the following diagram commutes:
\[
\xymatrix{
(d+1)\text{-}\operatorname{silt}\mathcal{P} \ar[rr]^-{\psi}_-{\cong }\ar[dr]_-{\phi }^{\cong } && \operatorname{h.cotors}\,\mathcal{P}^{[0,d]} \ar[dl]^-{\tau_{\geq -d+1}}_{\cong } \\
& \operatorname{f.}s\operatorname{-tors}\mathcal{H}_{\mathcal{P}}^d &
}
\]
\end{theorem}
\begin{proof}
Note that $(d+1)\text{-}\operatorname{conf. silt}\mathcal{P}=(d+1)\text{-}\operatorname{silt}\mathcal{P}$, since all $(d+1)$-term silting subcategories with respect to $\mathcal{P}$ are contravariantly finite in $\mathcal{D}$.
By Propositions \ref{dsisohc} and \ref{monofs}, we have the following commutative diagram
\[
\xymatrix{
(d+1)\text{-}\operatorname{silt}\mathcal{P} \ar[rr]^-{\psi}_{\cong }\ar@{>->}[dr]_-{\phi } && \operatorname{h.cotors}\mathcal{P}^{[0,d]} \ar[dl]^-{\tau_{\geq -d+1}} \\
& \operatorname{f.}s\operatorname{-tors}\mathcal{H}_{\mathcal{P}}^d &
}
\]
It suffices to show that
$
\tau_{\geq -d+1}:\operatorname{h.cotors}\mathcal{P}^{[0,d]}\rightarrow \operatorname{f.}s\operatorname{-tors}\mathcal{H}_{\mathcal{P}}^d
$
is an epimorphism, which implies that $\phi$ and $\tau_{\geq -d+1}$ are isomorphisms.

For any $\mathcal{T} \in \operatorname{f.}s\operatorname{-tors}\mathcal{H}_{\mathcal{P}}^d$,
define
\[
\varphi(\mathcal{T}) := \left\{ V \in \mathcal{P}^{[0,d]} \mid \tau_{\geq -d+1} V \in \mathcal{T} \right\}.
\]
Since $\tau_{\geq -d+1} \colon \mathcal{P}^{[0,d]} \to \mathcal{H}_{\mathcal{P}}^d$ is dense, we have $\tau_{\geq -d+1}(\varphi(\mathcal{T})) = \mathcal{T}$. In what follows, we will prove $\varphi(\mathcal{T}) \in \operatorname{h.cotors}\mathcal{P}^{[0,d]}$.

\medskip
\noindent\textbf{Step 1.} $\varphi(\mathcal{T})$ is closed under extensions.

Let $V_1 \stackrel{f}{\longrightarrow} X \stackrel{}{\longrightarrow} V_2 \stackrel{}{\longrightarrow} V_1[1]$
be a triangle in $\mathcal{D}$ with $V_1,\,V_2 \in \varphi(\mathcal{T})$.
By Lemma \ref{tauCtau}, we have
$\tau_{\geq -d+1} C(\tau_{\geq -d+1} f) \cong \tau_{\geq -d+1} V_2 \in \mathcal{T}.$ Let $\mathcal{F}$ be the torsion-free class corresponding to $\mathcal{T}$.
For any $F\in \mathcal{F} $, applying $\operatorname{Hom}(-,F)$ to the triangle:
\[
\tau_{\leq -d} C(\tau_{\geq -d+1} f) \to C(\tau_{\geq -d+1} f) \to \tau_{\geq -d+1} C(\tau_{\geq -d+1} f)\to\tau_{\leq -d} C(\tau_{\geq -d+1} f)[1],
\]
we get $\operatorname{Hom}(C(\tau_{\geq -d+1} f), F) = 0$.
Applying $\operatorname{Hom}(-,F)$ to the triangle
\[
\tau_{\geq -d+1} V_1 \xrightarrow{\tau_{\geq -d+1} f} \tau_{\geq -d+1} X \stackrel{}{\longrightarrow} C(\tau_{\geq -d+1} f)\stackrel{}{\longrightarrow}\tau_{\geq -d+1} V_1[1],
\]
we get $\tau_{\geq -d+1} X \in {}^\perp\mathcal{F} \cap \mathcal{H}_{\mathcal{P}}^d = \mathcal{T}$,
i.e. $X \in \varphi(\mathcal{T})$.

\medskip
\noindent\textbf{Step 2.} $\varphi(\mathcal{T})$ is a cotorsion class in $\mathcal{P}^{[0,d]}$.

Since $\mathcal{P}[d] \subseteq \mathcal{P}^{\perp_{>-d}} = \mathcal{D}^{\leq -d}$, we have $\tau_{\geq -d+1}(\mathcal{P}[d]) = 0$, and then $\mathcal{P}[d] \subseteq \varphi(\mathcal{T})$.
Since $\mathcal{T}$ is covariantly finite and $\tau_{\geq -d+1}(\varphi(\mathcal{T})) = \mathcal{T}$,
by Corollary \ref{covfequ} we get $\varphi(\mathcal{T})$ is covariantly finite in $\mathcal{P}^{[0,d]}$.

For any $R \in \mathcal{P}^{[0,d]}$, take a minimal left $\varphi(\mathcal{T})$-approximation $v \colon R \to V_R$ of $R$. Since $R,\,V_R \in \mathcal{P}^{[0,d]}$, we have the triangles in the two rows of the following diagram
\begin{equation}\label{zjfk}
\xymatrix{
R^{[0,d-1]} \ar[r]^-{\iota_R} & R \ar[r]^-{\pi_R} \ar[d]^{v} & R^d[d]\ar[r]\ar@{-->}[d]^-{\alpha[d]} & R^{[0,d-1]}[1]\\
V_R^{[0,d-1]} \ar[r]^-{\iota_V} & V_R \ar[r]^-{\pi_V} & V_R^d[d]\ar[r] & R^{[0,d-1]}[1]
}\end{equation}
with $R^{[0,d-1]},\,V_R^{[0,d-1]} \in \mathcal{P}^{[0,d-1]}$ and $R^d,\,V_R^d \in \mathcal{P}$.
Noting that $\operatorname{Hom}(R^{[0,d-1]}, V_R^d[d]) = 0$, we have $\pi_V\circ v\circ\iota_R=0$.
Thus, there exists $\alpha \colon R^d \to V_R^d$ such that the middle square in \eqref{zjfk} commutes.

By \cite[Remark 2.11(2)]{NP19} and the $3\times 3$ lemma, we obtain the following diagram of triangles
\[
\xymatrix{
R^{[0,d-1]} \ar[r] \ar[d] & R \ar[r]^-{\pi_R} \ar[d]^-{\binom{\pi_R}{v}}  & R^d[d] \ar[d]^-{\binom{id}{\alpha[d]}}\ar[r]            &R^{[0,d-1]}[1]\ar[d] \\
V_R^{[0,d-1]} \ar[r]\ar[d]& R^d[d] \oplus V_R \ar[r]^-{\left(\begin{smallmatrix}id&0\\0&\pi_V\end{smallmatrix}\right)}\ar[d]     & R^d[d] \oplus V_R^d[d] \ar[d]^{(-\alpha[d]~id)}\ar[r] &V_R^{[0,d-1]}[1]\ar[d] \\
Q \ar[r]\ar[d]          & C(\binom{\pi_R}{v}) \ar[r]\ar[d]                           & V_R^d[d]\ar[r]\ar[d]&Q[1]\ar[d]\\
R^{[0,d-1]}[1] \ar[r]  & R[1] \ar[r]^-{\pi_R[1]}   & R^d[d+1] \ar[r]             &R^{[0,d-1]}[2].
}
\]
Since $Q,\,V_R^d[d] \in \mathcal{P}^{[0,d]}$, we have $C(\binom{\pi_R}{v}) \in \mathcal{P}^{[0,d]}$. Noting that $R^d[d] \in \mathcal{P}[d] \subseteq \varphi(\mathcal{T})$, we obtain the morphism $\binom{\pi_R}{v}$ is also a left $\varphi(\mathcal{T})$-approximation of $R$.

By the Krull-Schmidt property of $\mathcal{D}$, there exists the following isomorphism of triangles
\[
\xymatrix{
R \ar[r]^-{\binom{\pi_R}{v}} \ar@{=}[d] & R^d[d] \oplus V_R \ar[r] \ar[d]^\cong & C(\binom{\pi_R}{v}) \ar[d]^\cong\ar[r]&R[1]\ar@{=}[d] \\
R \ar[r]^-{\binom{0}{v}} & R^d[d] \oplus V_R \ar[r]^{\left(\begin{smallmatrix}id&0\\0&\pi_V\end{smallmatrix}\right)} &R^d[d] \oplus C(v)\ar[r]&R[1].
}
\]
Thus $C(v) \in \mathcal{P}^{[0,d]}$. By Lemma \ref{Wakamatsu}, we conclude $C(v) \in {^{\perp_1}\varphi(\mathcal{T})}$.

On the other hand, we have a triangle
$$R^{[1,d]}[-1] \stackrel{}{\longrightarrow} R^0\stackrel{}{\longrightarrow} R \stackrel{}{\longrightarrow} R^{[1,d]}$$
such that $R^{[1,d]}\in\mathcal{P}^{[1,d]}$ and $R^0\in\mathcal{P} $.
Since $R^{[1,d]}[-1] \in \mathcal{P}^{[0,d]}$, by the above arguments, we have a triangle
$$R^{[1,d]}[-1] \stackrel{}{\longrightarrow} V \stackrel{}{\longrightarrow} U \stackrel{}{\longrightarrow} R^{[1,d]}$$
such that $V\in\varphi(\mathcal{T})$ and $U\in{^{\perp_1}\varphi(\mathcal{T})}\cap\mathcal{P}^{[0,d]}$.
Applying the octahedral axiom, we obtain the following commutative diagram of triangles
\[
\xymatrix{
R[-1] \ar[r] \ar@{=}[d] & R^{[1,d]}[-1] \ar[r] \ar[d] & R^0 \ar[r] \ar[d] & R \ar@{=}[d] \\
R[-1] \ar[r] & V \ar[r] \ar[d] & W \ar[r] \ar[d] & R\ar[d] \\
& U \ar@{=}[r]\ar[d] & U\ar[r]\ar[d]&R^{[1,d]}\\
&R^{[1,d]}\ar[r]&R^0[1].&
}
\]
Since $R^0 \in \mathcal{P} \subseteq {^{\perp_1}\mathcal{P}^{[0,d]}} \subseteq {^{\perp_1}\varphi(\mathcal{T})}$ and $U \in {^{\perp_1}\varphi(\mathcal{T})}\cap\mathcal{P}^{[0,d]}$,
we have $W \in {^{\perp_1}\varphi(\mathcal{T})}\cap\mathcal{P}^{[0,d]}$. Hence, $({^{\perp_1}\varphi(\mathcal{T})}\cap\mathcal{P}^{[0,d]}, \varphi(\mathcal{T}))$ is a cotorsion pair in $\mathcal{P}^{[0,d]}$.

\medskip
\noindent\textbf{Step 3.} The cotorsion pair $({^{\perp_1}\varphi(\mathcal{T})}\cap\mathcal{P}^{[0,d]}, \varphi(\mathcal{T}))$ is hereditary.

By Lemma \ref{hequs}, it suffices to show that $\varphi(\mathcal{T})$ is closed under cones.
Let
\[
V_1 \stackrel{v}{\longrightarrow} V_2 \stackrel{}{\longrightarrow} X \stackrel{}{\longrightarrow} V_1[1]
\]
be a triangle with $V_1,V_2\in\varphi(\mathcal{T})$ and $X\in\mathcal{P}^{[0,d]}$.
For any $F\in\mathcal{F}$, applying $\operatorname{Hom}(-,F)$ to the triangle
\[
\tau_{\geq -d+1}V_1\stackrel{\tau_{\geq -d+1}v}{\longrightarrow} \tau_{\geq -d+1}V_2\stackrel{}{\longrightarrow}C(\tau_{\geq -d+1}v)\stackrel{}{\longrightarrow}\tau_{\geq -d+1}V_1[1],
\]
we obtain an exact sequence
\[
\operatorname{Hom}(\tau_{\geq -d+1}V_1[1], F)
\to \operatorname{Hom}(C(\tau_{\geq -d+1}v), F)
\to \operatorname{Hom}(\tau_{\geq -d+1}V_2, F)=0.
\]
Noting that $(\mathcal{T},\mathcal{F})$ is an $s$-torsion pair in $\mathcal{H}_{\mathcal{P}}^d$,
we have $\operatorname{Hom}(\tau_{\geq -d+1}V_1[1], \mathcal{F}) = 0$
and then $\operatorname{Hom}(C(\tau_{\geq -d+1}v), F)=0$.
Applying $\operatorname{Hom}(-,F)$ to the triangle
\[
\tau_{\leq -d} C(\tau_{\geq -d+1}v) \to C(\tau_{\geq -d+1}v) \to \tau_{\geq -d+1} C(\tau_{\geq -d+1}v)\to\tau_{\leq -d} C(\tau_{\geq -d+1}v)[1],
\]
we obtain an exact sequence
\[
\operatorname{Hom}(\tau_{\leq -d} C(\tau_{\geq -d+1}v)[1], F)
\to \operatorname{Hom}(\tau_{\geq -d+1} C(\tau_{\geq -d+1}v), F)
\to \operatorname{Hom}(C(\tau_{\geq -d+1}v), F).
\]
Since $\operatorname{Hom}(\tau_{\leq -d} C(\tau_{\geq -d+1}v)[1], F)=0$ and $\operatorname{Hom}(C(\tau_{\geq -d+1}v), F) = 0$, by Lemma \ref{tauCtau},
we get $\tau_{\geq -d+1} X \cong \tau_{\geq -d+1} C(\tau_{\geq -d+1}v) \in \mathcal{T},$
i.e. $X\in\varphi(\mathcal{T}).$ Hence, we finish the proof.
\end{proof}

\section{Applications to DG algebras}
In this section, let $A$ be a non-positive dg algebra over a field $K$ whose total cohomology $H^*(A)$ is finite dimensional. Denote by $D(A)$ the derived category of (right) dg $A$-modules.
Let $D_{\rm fd}(A)$ be the full subcategory of $D(A)$ consisting of dg
$A$-modules $M$ with finite-dimensional total cohomology $H^\ast(M)$.
Then $A_A \in D_{\rm fd}(A)$ and $D_{\rm fd}(A)$ admits a standard $t$-structure
$(D_{\rm fd}(A)^{\leq 0}, D_{\rm fd}(A)^{\geq 1})$, where
\[
D_{\rm fd}(A)^{\leq 0} := \{ M \in D_{\rm fd}(A) \mid H^i(M) = 0 \text{ for all } i>0 \} = A^{\perp_{>0}}
\]
and
\[
D_{\rm fd}(A)^{\geq 1} := \{ M \in D_{\rm fd}(A) \mid H^i(M) = 0 \text{ for all } i \leq 0 \} = A^{\perp_{\leq 0}}.
\]
The heart of this $t$-structure is equivalent to the category $\mathrm{mod}(H^0(A))$ of finite dimensional modules over $H^0(A)$.
By \cite[Proposition 6.12]{AMY19}, ${D}_{\rm fd}(A)$ is Hom-finite.
Note that $D_{\rm fd}(A)$ is closed under direct summands in ${D}(A)$ and thus ${D}_{\rm fd}(A)$ is idempotent complete. Hence, ${D}_{\rm fd}(A)$ is a Hom-finite Krull-Schmidt triangulated category.
Therefore, the subcategories $\mathcal{H}_A^d:=D_{\rm fd}(A)^{\leq 0} \cap {D}_{\rm fd}(A)^{\geq -d+1}$ and $\mathrm{per}(A):={\bf thick}_{D_{\rm fd}(A)}(A)$ of ${D}_{\rm fd}(A)$ are also Krull-Schmidt.

Since $A$ is a silting object of $\mathrm{per}(A)$, i.e. the additive closure $\mathrm{add}(A)$ of $A$ in $\mathrm{per}(A)$ is a
silting subcategory of $\mathrm{per}(A)$, by \cite[Proposition 2.20]{AI12},
any silting subcategory of $\mathrm{per}(A)$ is of the form $\mathrm{add}(T)$ for
some object $T\in\mathrm{per}(A)$, and thus is contravariantly finite in ${D}_{\rm fd}(A)$.
Since $H^*(A)$ is finite dimensional, there exists a positive integer $d$ such that
$A \in \mathcal{H}_A^d$.
Set $(d+1)$-$\mathrm{silt}\,A := (d+1)$-$\mathrm{silt}\,\mathrm{add}(A)$.
As an application of Theorem \ref{sfbij}, we have the following.

\begin{corollary}
There are isomorphisms of posets such that the following diagram commutes
\[
\xymatrix{
(d+1)\text{-}\mathrm{silt}\,A \ar[rr]^{\cong} \ar[rd]_{\cong} & & \operatorname{h.cotors}\,\mathrm{add}(A)^{[0,d]} \ar[ld]^{\cong} \\
& \operatorname{f.}s\operatorname{-tors}\mathcal{H}_A^d &
}
\]
\end{corollary}

For convenience, let us briefly recall the construction of $\mathrm{add}(A)$-presentation for any object in $\mathcal{H}_A^d$.
One can refer to \cite[Lemma 2.3]{WZ25} and \cite[Lemma 2.7]{MP25} for more details.

For any $M \in \mathcal{H}_A^d$ and $0\leq i\leq d$, we have the following triangles in ${D}_{\rm fd}(A)$:
\[
\Omega^{i+1}(M) \stackrel{\iota^i}{\longrightarrow} P^i \stackrel{\pi^i}{\longrightarrow} \Omega^i(M) \stackrel{h^i}{\longrightarrow}\Omega^{i+1}(M)[1],
\]
where $\Omega^0(M)=M$ and $\pi^i$ is the minimal right $\mathrm{add}(A)$-approximation of $\Omega^i(M)$.

Let $h:= h^d[d] \circ \cdots \circ h^1[1] \circ h^0:M \to \Omega^{d+1}(M)[d+1]$.
According to the proof of \cite[Lemma 2.7]{MP25}, applying the octahedral axiom inductively, we finally obtain the following commutative diagram of triangles
\begin{equation}\label{ppst}
\xymatrix{& P^{[0,d-1]} \ar@{=}[r] \ar[d] & P^{[0,d-1]} \ar[d] \\
\Omega^{d+1}(M)[d] \ar[r] \ar@{=}[d] & p(M) \ar[r] \ar[d] & M \ar[r]^-h\ar[d] & \Omega^{d+1}(M)[d+1] \ar@{=}[d] \\
\Omega^{d+1}(M)[d] \ar[r]^-{\iota^d[d]} & P^d[d] \ar[r]^-{\pi^d[d]}\ar[d] & \Omega^d(M)[d] \ar[r]^-{h^d[d]}\ar[d] & \Omega^{d+1}(M)[d+1]\ar[d]\\
&P^{[0,d-1]}[1] \ar@{=}[r] & P^{[0,d-1]}[1]\ar[r]&p(M)[1]}\end{equation}
where $P^{[0,d-1]}\in \mathrm{add}(A)^{[0,d-1]}$ and $p(M)$ is an $\mathrm{add}(A)$-presentation of $M$.
\begin{remark}
We claim that the $p(M)$ in \eqref{ppst} does not have $P[d]$ as direct summands for
any nonzero $P \in \mathrm{add}(A)$.
Otherwise, since $\mathrm{Hom}(P^{[0,d-1]}, P[d]) = 0$ and $\mathrm{Hom}(P[d], M) = 0$,
the morphisms $p(M) \to P^d[d]$ and $\Omega^{d+1}(M)[d] \to p(M)$ in \eqref{ppst} have $P[d] \xrightarrow{id} P[d]$ as direct summands.
Then $\iota^d[d]$ has $P[d] \xrightarrow{id} P[d]$ as direct summands.
This contradicts the right minimality of $\pi^d[d]$.
Furthermore, by Proposition \ref{PequH}, $p(M)$ is unique up to isomorphism.
For any $X\in D_{\rm fd}(A)$, denote by $|X|$ the number of non-isomorphic indecomposable direct summands of $X$. Then we have $|p(M)| = |M|$. In what follows, for each $M \in \mathcal{H}_A^d$, we fix the notation $p(M)$ for the $\mathrm{add}(A)$-presentation of $M$ given in \eqref{ppst}.
\end{remark}

An object $P$ of $\mathcal{H}_A^d$ is called \textit{projective} if $\mathrm{Hom}(P, M[1]) = 0$ for any $M \in \mathcal{H}_A^d$.
Dually, an object $I$ of $\mathcal{H}_A^d$ is called \textit{injective} if $\mathrm{Hom}(M, I[1]) = 0$ for any $M \in \mathcal{H}_A^d$.
Denote by $\mathcal{P}(\mathcal{H}_A^d)$ and $\mathcal{I}(\mathcal{H}_A^d)$ the full subcategories of projective and injective objects in $\mathcal{H}_A^d$, respectively.
Let $\underline{\mathcal{H}_A^d}:=\mathcal{H}_A^d/\mathcal{P}(\mathcal{H}_A^d)$ and $\overline{\mathcal{H}_A^d}:=\mathcal{H}_A^d/\mathcal{I}(\mathcal{H}_A^d)$.
The category equivalence $\tau:\underline{\mathcal{H}_A^d}\to\overline{\mathcal{H}_A^d}$ defined in \cite{MP25} maps $M$ to $\tau_{\leq 0}(\nu (p(M))[-1])$, where $\nu:D_{\rm fd}(A)\to D_{\rm fd}(A)$ denotes the derived Nakayama functor. Analogous to \cite[Lemma 4.8]{Z24}, we have the following.

\begin{lemma}\label{moriso}
Let $M,N \in \mathcal{H}_A^d$ and $P \in \mathrm{add}(A)$. For any $i\geq 1$, there are isomorphisms
\[
\mathrm{Hom}(p(M), p(N)[i]) \cong \mathrm{Hom}(p(M), N[i])\cong \mathbb{D}\mathrm{Hom}(N, \tau(M)[1-i]),
\]
where $\mathbb{D}:=\mathrm{Hom}_K(-,K)$ denotes the $K$-linear duality.
\end{lemma}

Following \cite{Z24}, we generalize the $\tau$-tilting theory in $\mathrm{mod}(H^0(A))$ to $\mathcal{H}_A^d$.

\begin{definition}\cite[Definition 1.13]{Z24}
Let $\mathcal{M}$ be a subcategory of $\mathcal{H}^d_A$, and let $n$ be a positive integer.
Suppose there exist extriangles in $\mathcal{H}^d_A$
\begin{equation}\label{ejiao}
X_{i+1}\stackrel{}{\longrightarrow}M_{i+1}\stackrel{}{\longrightarrow}X_i\stackrel{}\dashrightarrow\end{equation}
with $M_{i+1} \in \mathcal{M}$, $0\leq i\leq n-1$. Then $X_0$ is called an {\em $n$-factor} of $\mathcal{M}$.
We denote by $\mathrm{Fac}_n(\mathcal{M})$ the subcategory of $\mathcal{H}_A^d$ consisting of all $n$-factors of $\mathcal{M}$.
\end{definition}

\begin{definition}\cite[Definitions 4.1, 4.4 and 4.5]{Z24}
  {\rm (a)} A dg $A$-module $M\in\mathcal{H}_A^d$ is called \textit{positive $\tau$-rigid} if
 $\mathrm{Hom}(M, \tau(M)[\leq 0]) = 0.$

  {\rm (b)} A pair $(M,P)$ with $M \in \mathcal{H}_A^d$ and $P \in \mathrm{add}(A)$ is called a \textit{positive $\tau$-rigid pair}
  if $M$ is positive $\tau$-rigid and $\mathrm{Hom}(P, M[\leq 0]) = 0.$

  {\rm (c)} A positive $\tau$-rigid pair $(M,P)$ is called \textit{$\tau$-tilting} if
\[
{}^{\perp_{\leq 0}}(\tau(M)) \cap P^{\perp_{\leq 0}} \cap \mathcal{H}_A^d =\mathrm{Fac}_d(\mathrm{add}(M))=:\mathrm{Fac}_d(M).
\]

An object in $\mathcal{H}_A^d$ is called {\em basic}, if its indecomposable direct summands are pairwise non-isomorphic.
A positive $\tau$-rigid pair or $\tau$-tilting pair $(M,P)$ is said to be {\em basic}, if $M$ and $P$ are basic.
We denote by $\tau\text{-}\mathrm{tiltp}\,\mathcal{H}_A^d$ the set consisting of all basic $\tau$-tilting pairs in $\mathcal{H}^d_A$.
For any $(M_1,P_1),(M_2,P_2)\in \tau\text{-}\mathrm{tiltp}\,\mathcal{H}_A^d$, define $(M_1,P_1)\leq(M_2,P_2)$ if $\mathrm{Fac}_d(M_1)\subseteq \mathrm{Fac}_d(M_2)$.
\end{definition}

\begin{remark}\label{rigidtotilt}
  Let $(M,P)$ be a positive $\tau$-rigid pair of $\mathcal{H}_A^d$.
  By \cite[Lemma 1.19]{Z24}, we have
  $\mathrm{Fac}_d(M)\subseteq {^{\perp_{\leq 0}}(\tau(M))}.$
  By the definition of $\mathrm{Fac}_d(M)$, applying $\operatorname{Hom}(P,-)$ to the extriangles in \eqref{ejiao},
  we get $\mathrm{Fac}_d(M)\subseteq P^{\perp_{[-d+1,0]}}.$
  Since $\mathrm{Fac}_d(M)\subseteq \mathcal{H}_A^d\subseteq A^{\perp_{\leq -d}},$
  we have $\mathrm{Fac}_d(M)\subseteq{}^{\perp_{\leq 0}}(\tau(M)) \cap P^{\perp_{\leq 0}} \cap \mathcal{H}_A^d.$
  Thus, a positive $\tau$-rigid pair $(M,P)$ is $\tau$-tilting if and only if
  ${}^{\perp_{\leq 0}}(\tau(M)) \cap P^{\perp_{\leq 0}} \cap \mathcal{H}_A^d \subseteq \mathrm{Fac}_d(M).$
\end{remark}

\begin{definition}\cite[Definition 2.7]{WZ25}
  Let $\mathcal{M}$ be a subcategory of $\mathcal{H}_A^d$.
  If it admits an $\mathrm{add}(A)$-presentation $\mathcal{S}$ such that
 $\mathcal{S}^{\perp_{>0}}\cap \mathcal{H}_A^d=\mathrm{Fac}_d(\mathcal{M}),$
  then $\mathcal{M}$ is called an \textit{AIR tilting} subcategory of $\mathcal{H}_A^d$ with respect to $\mathcal{S}$.
  We denote by $\mathrm{AIR}\text{-}\mathrm{tilt}\,\mathcal{H}_A^d$ the collection consisting of all AIR tilting subcategories of $\mathcal{H}^d_A$.
  Similar to $\tau\text{-}\mathrm{tiltp}\,\mathcal{H}_A^d$, a partial order relation can be defined on $\mathrm{AIR}\text{-}\mathrm{tilt}\,\mathcal{H}_A^d$.
\end{definition}

\begin{lemma}\cite[Corollary 2.16]{WZ25}
  Let $\mathcal{M}$ be a subcategory of $\mathcal{H}_A^d$.
  If $\mathcal{M}$ is AIR tilting with respect to its two $\mathrm{add}(A)$-presentations $\mathcal{S}$ and $\mathcal{S}'$,
  then $\mathcal{S}=\mathcal{S}'$.
\end{lemma}

\begin{proposition}\label{tausiltbij}
There exist isomorphisms of posets such that the following diagram commutes
\begin{equation}\label{diyihang}
\xymatrix{
\mathrm{AIR}\text{-}\mathrm{tilt}\,\mathcal{H}_A^d \ar[rr]^{\cong} && (d+1)\text{-}\mathrm{silt}\,A \\
&\tau\text{-}\mathrm{tiltp}\,\mathcal{H}_A^d \ar[ru]_{\cong} \ar[lu]^{\cong} &
}
\end{equation}
\end{proposition}

\begin{proof}
The isomorphism in the first row of \eqref{diyihang} follows from \cite[Theorem 2.8]{WZ25},
which maps the AIR-tilting subcategories of $\mathcal{H}_A^d$ to their $\mathrm{add}(A)$-presentations.

For any pair $(M,P)$ with $M \in \mathcal{H}_A^d$ and $P \in \mathrm{add}(A)$, set $\mathbb{P}:= p(M) \oplus P[d]$.
By Proposition \ref{PequH}, we get $\mathrm{add}(\mathbb{P})$ is an $\mathrm{add}(A)$-presentation of $\mathrm{add}(M)$.

Note that for any $N \in \mathcal{H}_A^d= A^{\perp_{\leq -d}}\cap A^{\perp_{>0}}=A[d]^{\perp_{\leq 0}}\cap A[d]^{\perp_{>d}}$,
\[\mathrm{Hom}(P[d], N[i]) = 0 \text{ for any } i \geq 1\]
if and only if
\[\mathrm{Hom}(P[d], N[i]) = 0 \text{ for any } i \leq d\]
if and only if
\[\mathrm{Hom}(P, N[i]) = 0 \text{ for any } i \leq 0.\]
Thus, we have $P^{\perp_{\leq 0}}\cap \mathcal{H}_A^d = P[d]^{\perp_{>0}}\cap \mathcal{H}_A^d$.
By Lemma \ref{moriso}, we have ${}^{\perp_{\leq 0}}(\tau(M))\cap \mathcal{H}_A^d = p(M)^{\perp_{>0}}\cap \mathcal{H}_A^d$.
Hence, we get ${}^{\perp_{\leq 0}}(\tau M) \cap P^{\perp_{\leq 0}}{\cap \mathcal{H}_A^d} = \mathbb{P}^{\perp_{>0}}\cap \mathcal{H}_A^d$.
Moreover, if $(M,P)$ is a basic $\tau$-tilting pair,
then
\[\mathrm{Fac}_d(M) = {}^{\perp_{\leq 0}}(\tau M) \cap P^{\perp_{\leq 0}} \cap \mathcal{H}_A^d = \mathbb{P}^{\perp_{>0}} \cap \mathcal{H}_A^d.\]
It follows that the subcategory $\mathrm{add}(M)$ of $\mathcal{H}_A^d$ is AIR tilting with respect to its $\mathrm{add}(A)$-presentation $\mathrm{add}(\mathbb{P})$.

If $(M,Q)$ is also a basic $\tau$-tilting pair, then $(M,P\oplus Q)$ is a positive $\tau$-rigid pair.
Since $(M,Q)$ is a $\tau$-tilting pair, we obtain
\[{}^{\perp_{\leq 0}}(\tau(M)) \cap (P\oplus Q)^{\perp_{\leq 0}} \cap \mathcal{H}_A^d\subseteq{}^{\perp_{\leq 0}}(\tau(M)) \cap Q^{\perp_{\leq 0}} \cap \mathcal{H}_A^d\subseteq\mathrm{Fac}_d(M).\]
By Remark \ref{rigidtotilt}, $(M,P\oplus Q)$ is a $\tau$-tilting pair.
According to the above discussion, we obtain $p(M)\oplus P[d]\oplus Q[d]$ is a silting object of $\mathrm{per}(A)$.
By \cite[Corollary 2.28]{AI12}, we get
\[|p(M)\oplus (P\oplus Q)[d]| = |\mathbb{P}|=|p(M)\oplus P[d]|.\]
Since $p(M)$ does not have $P'[d]$ as direct summands for any nonzero $P' \in \mathrm{add}(A)$, we get $|P\oplus Q|=|P|$.
Hence, $Q$ is a direct summand of $P$. Similarly, $P$ is also a direct summand of $Q$.
Since $P$ and $Q$ are both basic, we conclude $P\cong Q$.
Moreover, since $M$ is basic, we get that the correspondence $(M,P)\mapsto \mathrm{add}(M)$ gives an injective map from $\tau\text{-}\mathrm{tiltp}\,\mathcal{H}_A^d$ to $\mathrm{AIR}\text{-}\mathrm{tilt}\,\mathcal{H}_A^d$, and thus an injective map from $\tau\text{-}\mathrm{tiltp}\,\mathcal{H}_A^d$ to $(d+1)\text{-}\mathrm{silt}\,A$ by combining with the isomorphism in the first row of \eqref{diyihang}.

Conversely, let $\mathcal{S}$ be a silting subcategory of $\mathrm{per}(A)$.
By \cite[Proposition 2.20]{AI12} and the Krull-Schmidt property of $\operatorname{per}(A)$,
we can take its unique basic additive generator $\mathbb{S} = S \oplus P[d]$, where $P \in \mathrm{add}(A)$ and $S$ does not have $P'[m]$ as direct summands for any nonzero $P' \in \mathrm{add}(A)$, then $S = p(\tau_{\geq -d+1}\mathbb{S})$.
Moreover,
noting that $\mathcal{S}$ is a $(d+1)$-term silting subcategory of $\mathrm{per}(A)$, by \cite[Proposition 2.12]{WZ25}, we get
\[
\tau_{\geq-d+1}\mathbb{S}\in\mathrm{Fac}_d(\tau_{\geq-d+1}\mathbb{S}) = \mathbb{S}^{\perp_{>0}}\cap \mathcal{H}_A^d = {}^{\perp_{\leq 0}}(\tau (\tau_{\geq -d+1}\mathbb{S})) \cap P^{\perp_{ \leq 0}} \cap \mathcal{H}_A^d.
\]
Hence, $(\tau_{\geq-d+1}\mathbb{S}, P)$ is a basic $\tau$-tilting pair, and then
$\mathrm{add}(\tau_{\geq-d+1}\mathbb{S})$ is the AIR tilting subcategory with respect to its $\mathrm{add}(A)$-presentation $\mathrm{add}(\mathbb{S})=\mathcal{S}$. Therefore, we complete the proof.
\end{proof}

\section*{Acknowledgments}
This work is partially supported by the National Natural Science Foundation of China (No. 12271257) and the Natural Science Foundation of Jiangsu Province of China (No. BK20240137).

\end{document}